\documentclass[11pt,letterpaper]{article}
\usepackage{ifthen,latexsym,amssymb,amsmath, tikz}

\newtheorem{theorem}{Theorem}
\newcommand{\bth}[2][nothing]{\ifthenelse{\equal{#1}{nothing}}
 {\begin{theorem}} {\begin{theorem}[#1]}\label{th:#2}}

\newtheorem{lemma}[theorem]{Lemma}
\newcommand{\blm}[2][nothing]{\ifthenelse{\equal{#1}{nothing}}
 {\begin{lemma}} {\begin{lemma}[#1]}\label{lm:#2}}

\newtheorem{cor}[theorem]{Corollary}
\newcommand{\bcor}[2][nothing]{\ifthenelse{\equal{#1}{nothing}}
 {\begin{cor}} {\begin{cor}[#1]}\label{co:#2}}

\newcommand{\ceil}[1]{\lceil #1\rceil}

  \newenvironment{proof}{\vspace{1ex}\noindent{\bf Proof:}}{\hspace*{\fill}
  $\blacksquare$\vspace{1ex}}
  \newenvironment{proofof}[1]{\vspace{1ex}\noindent{\bf Proof of #1:}}{\hspace*{\fill}
  $\blacksquare$\vspace{1ex}}

\begin{document}

\def\Tbest{T_{\rm bestmix}}
\def\Tmix{T_{\rm mix}}
\def\Tres{T_{\rm reset}}
\def\Tfor{T_{\rm forget}}
\def\Thit{T_{\rm hit}}
\def\kmp{\k(\mu, \pi)}

\def\Ret{\mbox{Ret}}

\def\rp{{\hat{p}}}
\def\rx{{\hat{x}}}
\def\rP{{\hat{P}}}
\def\rX{{\hat{X}}}
\def\rH{{\hat{H}}}

\def\rTbest{\hat{T}_{\rm bestmix}}
\def\rTmix{\hat{T}_{\rm mix}}
\def\rTres{\hat{T}_{\rm reset}}
\def\rTfor{\hat{T}_{\rm forget}}
\def\rThit{\hat{T}_{\rm hit}}

\def\rmu{\hat{\mu}}

\def\bH{{\overline{H}}}
\def\bp{{\overline{\pi}}}

\def\L{{\Delta}}

\def\vol{\mathrm{vol}}
\def\bone{\mathbf{1}}
\def\bzero{\mathbf{0}}
\def\bh{\mathbf{h}}
\def\bj{\mathbf{j}}
\def\br{\mathbf{r}}
\def\bv{\mathbf{v}}
\def\bi{\mathbf{i}}
\def\bx{\mathbf{x}}

\def\cG{\mathcal{G}}
\def\cL{\mathcal{L}}
\def\cX{\mathcal{X}}

\def \bbL{\mathbb{L}}
\def \bbG{\mathbb{G}}
\def \rbbG{\hat{\bbG}}

\title{A Hitting Time Formula for the Discrete Green's Function  } 
\author{Andrew Beveridge\footnote{Department of Mathematics, Statistics and Computer Science, Macalester College, St Paul, MN 55105. \texttt{abeverid@macalester.edu}}}

\date{}

\maketitle

\begin{abstract}
The discrete Green's function (without boundary) $\bbG$ is a pseudo-inverse of the combinatorial Laplace operator of a graph $G=(V,E)$. We reveal the intimate connection between Green's function and the theory of exact stopping rules for random walks on graphs. We give an elementary formula for Green's function in terms of state-to-state hitting times of the underlying graph. Namely,
$$
\bbG(i,j) = \pi_j \left( \sum_{k \in V} \pi_k H(k,j) - H(i,j) \right)
$$
where $\pi_i$ is the stationary distribution at vertex $i$ and $H(i,j)$ is the expected hitting time for a random walk starting from vertex $i$ to first reach vertex $j$.
This formula also holds for the digraph Laplace operator.

The most important characteristics of a stopping rule are its exit frequencies, which are the expected number of  exits of a given vertex before the rule halts the walk. We show that Green's function is, in fact, a matrix of exit frequencies  plus a rank one matrix.  
In the undirected case, we derive spectral formulas for Green's function and for some mixing measures arising from stopping rules.  Finally, we further explore the exit frequency matrix point-of-view, and discuss a natural generalization of Green's function for any distribution $\tau$ defined on the vertex set of the graph.
\medskip


\end{abstract}

\section{Introduction}

Let $G=(V,E)$ be a simple undirected graph on vertices $V= \{ 1, 2, \ldots , n \}$. (For notational convenience, we identify a vertex with its label.) We define the volume of $G$ to be $\vol(G) = \sum_{k \in V} \deg(k)$. Let $A$ be the adjacency matrix of $G$  and let $D=\mathrm{diag} (\deg(1), \ldots , \deg(n))$ be the diagonal matrix of degrees. The \emph{discrete Laplace operator} (cf. \cite{chung}) is the $n \times n$ matrix 
\[
\L = I - D^{-1} A.
\]
We can view $\L$ as a linear transformation $\L: V^* \rightarrow V^*$ where $V^*$ denotes the vector space of all real functions on $V$. The Laplace operator is a variant of the graph Laplacian $L = D \L = D - A$ and the normalized graph Laplacian $\cL =D^{1/2}\L D^{-1/2} =   I - D^{-1/2}AD^{-1/2}.$

The Laplace operator is directly related to random walks on $G$. The matrix $P=D^{-1}A$ is the \emph{transition matrix} for a simple random walk on $G$ since  $P_{ij} = 1/\deg(i)$ when $ij \in E$ and 0 otherwise. The matrices $\L=I-P$ and $P$ share the same eigenvectors, where the eigenvalue $\lambda_k$ of $\L$ corresponds to the eigenvalue $\lambda_k' = 1-\lambda_k$ of $P$. 
The Laplace operator $\L$  has rank $n-1$: the vector $\pi = ( \pi_1, \ldots, \pi_n)$ where $\pi_i = \deg(i)/ \vol(G)$ is a left eigenvector for eigenvalue $\lambda_0=0$, and the all-ones vector is a right eigenvector. With respect to $P$, these are eigenvectors for eigenvalue $\lambda_0' = 1$, with the following interpretations.
Having $\bone$ as a right eigenvector for $\lambda_0'=1$ captures the fact  that the  transition probabilities from state $i$ sum to one. Having $\pi$ as a left eigenvector means that $\pi$ is the stationary distribution for random walks on $G$.

The discrete Green's function $\bbG$ was introduced by Chung and Yau \cite{chung+yau}. This $n \times n$ matrix is the pseudo-inverse of $\L$ given by
\begin{equation}
\label{eqn:green-cons}
\begin{array}{rcl}
\bbG \L &=& I - \bone \pi^{\top}, \\
\bbG \bone &=& \bzero.
\end{array}
\end{equation}
 The second constraint guarantees the uniqueness of $\bbG$. 
Green's function has been computed for some special families of graphs, including the path, the hypercube \cite{chung+yau}, products of cycles \cite{ellis}, the complete graph and trees \cite{xu+yau}. Many of these values have been computed via spectral formulas. Recently, Xu and Yau \cite{xu+yau} developed a  formula using two counting invariants of graphs that involve sums and products over spanning linear subgraphs of $G$. 

We give a very simple formula for Green's function in terms of  hitting times for random walks on $G$. This formula is more tractable and versatile than the results described above. Our formula also holds for weighted, directed graphs (or digraphs, for short), so we introduce some digraph definitions before stating our main theorem.
  
 Let $G=(V,E,W)$ be a weighted digraph on vertices $V=\{ 1, 2, \cdots , n \}$, directed edge set $E$ and non-negative edge weights $W = \{ w_{ij} \mid 1 \leq i,j \leq n \}$, where  $w_{ij}=0$ whenever $ij \notin E$. 
The corresponding  transition probability matrix $P$ has entries
\[
P_{ij} = \frac{w_{ij}}{\sum_{k \in V} w_{ik}}.
\]
This definition is a natural generalization of the undirected case: when $G$
is an undirected graph whose edges have unit weight, this formula becomes $P = D^{-1}A$.
For a strongly connected, aperiodic digraph $G$, the Perron-Frobenius theorem guarantees that eigenvalue $\lambda=1$ has  a unique unit left  eigenvector  $\pi$ with $\pi(i) > 0$ for all $i \in V$. This vector $\pi$ is the stationary distribution for the random walk corresponding transition matrix $P$.
 Li and Zhang \cite{li+zhang} introduce the normalized digraph Laplacian $\cL = \Pi^{1/2} (I - P) \Pi^{-1/2}$,
where  $\Pi$ is the diagonal matrix with $\Pi_{ii} = \pi(i)$.  The  corresponding digraph Laplace operator is $\L = I - P.$ 
As in the undirected case, Green's function $\bbG$ is the matrix satisfying
$\bbG \L = I - \bone \pi^{\top}$ and $\bbG \bone = \bzero$. 

From here forward, we assume that $G=(V,E)$ is a strongly connected digraph, where we  break periodicity by considering a lazy random walk, if necessary.
 Given  $i,j \in V$, the \emph{hitting time} $H(i,j)$ is the expected number of steps before a random walk started at $i$ first reaches $j$. We choose to define $H(i,i)=0$ and let the \emph{return time} $\mbox{Ret}(i)$ denote the expected number of steps before a random walk started at $i$ first returns to $i$. We also define
\[
H(\pi,j) = \sum_{i \in V} \pi_i H(i,j)
\]
to be the expected number of steps it takes for a walk starting from a random initial vertex to reach $j$. We are now ready to state our main result.
\begin{theorem}
\label{thm:green}
Let $G$ be a strongly connected digraph on vertices $V=\{ 1, 2, \ldots , n \}$.
Green's function $\bbG$ for $G$ is the $n \times n$ matrix given by
\begin{equation}
\label{eq:green-formula}
\bbG(i,j) = \pi_j ( H(\pi,j) - H(i,j)).
\end{equation}
\end{theorem}
We give two proofs of this  result. In Section \ref{sec:psychic-proof}, we give a proof for undirected graphs that uses two well known hitting time identities. This argument has the advantage of being  short and self-contained. However, this proof does not shed much light on \emph{why} this formula is correct.

We remedy this situation in Section \ref{sec:efm-proof}, where we give  a second proof of Theorem \ref{thm:green} that places  the result in a much richer context. This argument  holds for strongly connected directed graphs. Formula \eqref{eq:green-formula}  is a manifestation of the deep connection between  Green's function and  the theory of optimal stopping rules for random walks on $G$, introduced by Lov\'asz and Winkler \cite{lovasz+winkler, lovasz+winkler2}. Given a starting distribution $\sigma$ and a target distribution $\tau$, a \emph{stopping rule} $\Gamma(\sigma,\tau)$ halts a random walk whose initial vertex is drawn from $\sigma$ so that the final state is governed by $\tau$. In particular, we are naturally interested in optimal mixing rules $\Gamma(i,\pi)$ where $\sigma=i$ is a singleton distribution and $\pi$ is the stationary distribution.
In Section \ref{sec:efm-proof}, we review the basics about stopping rules on graphs. We then show that the matrix $\bbG$ is an expression of the vertex-wise characteristics of  the family of optimal mixing rules $\{ \Gamma(i,\pi) \}_{1 \leq i \leq n}$. In particular, we prove that $\bbG$ is a slight alteration of $X_{\pi}$, the \emph{exit frequency matrix} for $\pi$,  introduced in \cite{beveridge+lovasz}.

In Section \ref{sec:examples},  we use equation \eqref{eq:green-formula} to calculate $\bbG$ for some families of undirected graphs.
In Section \ref{sec:spectral}, we  develop  spectral formulas for Green's function and for the following three exact mixing measures. 
Let $H(i,\pi)$ denote the expected length of an optimal stopping rule from $i$ to $\pi$.   The \emph{mixing time} for a graph $G$ is
\[
\Tmix(G) = \max_{i\in V} H(i, \pi).
\] 
In other words,  $\Tmix(G)$ is the expected length of an optimal stopping rule to $\pi$ when we start from the worst possible initial vertex. The \emph{reset time} is the average mixing time:
\[
\Tres(G) = \sum_{i \in V} \pi_i H(i, \pi). 
\]
Finally, the \emph{hit time} is the expected hitting time between two states drawn from the stationary
distribution
\[
\Thit(G)  = \sum_{i,k \in V} \pi_i \pi_k H(i,k) =  \sum_{k \in V}  \pi_k H(\pi,k).
\]
The   \emph{random target identity} \cite{aldous+fill} captures the very useful phenomenon
\begin{equation}
\label{eqn:thit}
 \sum_{k \in V}  \pi_k H(i,k) =  \Thit(G) \quad \mbox{for all } i \in V.
\end{equation}
In our context, this means that the expected length of the naive rule ``draw a target vertex $k$ accord to distribution $\pi$, then perform a random walk until reaching $k$'' is independent of the starting vertex.
 The relationships between $\Tmix$, $\Tres$, $\Thit$ and other mixing measures of a graph are thoroughly explored in \cite{ALW}. Our spectral formulas are most conveniently stated using the eigenvectors for the matrix $\bbL = D^{-1/2} L D^{1/2} = D^{-1} \cL D$. 

\begin{theorem}
\label{thm:green-spectral}
Let $G$ be an undirected graph.
Let $0 = \lambda_0 < \lambda_1 \leq \cdots \leq \lambda_{n-1}$ be the eigenvalues for 
$\bbL = D^{-1/2} L D^{1/2}$ with corresponding orthonormal eigenvectors $\phi_0, \phi_1, \ldots, \phi_{n-1}.$ Then we have the following spectral formulas:
\begin{align}
\label{eqn:green-spectral}
\bbG(i,j) 
&=
\sqrt{\frac{\deg(j)}{\deg(i)}} \,  \sum_{k=1}^{n-1} \frac{1}{\lambda_k} \, {\phi_{ki} \phi_{kj}}, \\
\label{eqn:mix-spectral}
\Tmix & = \max_{i \in V} \, -  \frac{\vol(G)}{ \sqrt{\deg(i)\deg(i')}} \, 
\sum_{k=1}^{n-1} \frac{1}{\lambda_k} 
 \phi_{ki} \phi_{k i'}, \\
\label{eqn:reset-spectral}
\Tres 
&=  \, -   \sum_{i \in V}  \sqrt{\frac{\deg(i)}{\deg(i')}}\, 
\sum_{k=1}^{n-1} \frac{1}{\lambda_k} \,
 \phi_{ki} \phi_{k i'}, \\
\label{eqn:hit-time-spectral}
\Thit &= 
\sum_{k=1}^{n-1} \frac{1}{\lambda_k},
 \end{align}
where $i' \in V$ is a vertex satisfying $H(i',i) = \max_{j \in V} H(j,i).$
\end{theorem}
The spectral formula \eqref{eqn:hit-time-spectral} for $\Thit$ is well known, but we include it here for comparison.  Ellis \cite{ellis} gives an formula analogous to equation \eqref{eqn:green-spectral} for the normalized Green's function $\cG =  D^{1/2} \bbG D^{-1/2}$ in terms of the eigensystem for the normalized Laplacian $\cL$.

Finally, in Section \ref{sec:green-general}, we  generalize Green's function, based on the exit frequency matrix results in \cite{beveridge+lovasz}. Green's function is intimately related to optimal stopping rules from singleton distributions to the stationary distribution $\pi$. We can replace $\pi$ with any  distribution $\tau$ to define a comparable matrix $\bbG_{\tau}$ for that target distribution. We describe some duality results for these  matrices that suggest some future research directions.

\section{Proof for undirected graphs via cycle reversing}
\label{sec:psychic-proof}

We give a direct proof of Theorem \ref{thm:green} for an undirected graph $G=(V,E)$ that uses  some fundamental identities for  hitting times on undirected graphs. These identities will be helpful when we calculate examples in Section \ref{sec:examples}. 
First, we need the cycle reversing identity \cite{ctw}: for all $i,j,k \in V$, we have
\begin{equation}
\label{eqn:cri}
H(i,j) + H(j,k) + H(k,i) = H(j,i) + H(i,k) + H(k,j).
\end{equation}
We get another identity by multiplying equation \eqref{eqn:cri} by $\pi_k$, summing over $1 \leq k \leq n$ and using the random target identity \eqref{eqn:thit}: 
\begin{equation}
\label{eqn:cri-pi}
H(\pi,i) + H(i,j) = H(\pi,j) + H(j,i).
\end{equation}
Finally, it is well known that the return time to $j$ satisfies
\begin{equation}
\label{eqn:ret}
\Ret(j) = \frac{1}{\pi_j} = \frac{\vol(G)}{\deg(j)}.
\end{equation}
We can now give the first proof of our main result.

\begin{proofof}{Theorem \ref{thm:green} for an undirected graph}
Define the $n \times n$ matrix $B$ where $B_{ij} = \pi_j ( H(\pi,j) - H(i,j))$. We show that $B$ satisfies the constraint equations \eqref{eqn:green-cons}. This will confirm that $\bbG=B$. We check the second condition first. We have
\[
B \bone = \sum_j   \pi_j ( H(\pi,j) - H(i,j)) =  \sum_k \pi_k \sum_j  \pi_j   H(k,j) - \sum_j \pi_j H(i,j) = 0 
\]
by the random target identity \eqref{eqn:thit}. As for the first constraint, we have 
\begin{align*}
(B\L)_{ij} &= 
 (B - B D^{-1} A)_{ij} \\  
&=  B_{ij} - \sum_{k \sim j} \frac{1}{\deg(k)} B_{ik}\\
&=
\pi_j (H(\pi,j) - H(i,j))
- \sum_{k \sim j}  \frac{1}{\deg(k)} \pi_k (H(\pi,k) - H(i,k)) \\
&=
\frac{1}{\vol(G)} \left(
\sum_{k \sim j} (H(\pi,j) - H(\pi,k)) - \sum_{k \sim j} (H(i,j) - H(i,k)) 
\right) \\
&=
\frac{1}{\vol(G)} \left(
\sum_{k \sim j} (H(k,j) - H(j,k)) - \sum_{k \sim j} (H(i,j) - H(i,k)) 
\right) 
\\
&= \frac{1}{\vol(G)} \sum_{k \sim j} (H(k,i) -H(j,i)) \\
&= \pi_j \left( \sum_{k \sim j} \frac{1}{\deg(j)}H(k,i) - H(j,i) \right) \\
&=
\left\{
\begin{array}{cc}
-\pi_j & j \neq i \\
\pi_j (\mbox{Ret}(j) - 1)& j = i \\
\end{array}
\right. \\
&=
\left\{
\begin{array}{cc}
-\pi_j & j \neq i \\
1- \pi_j & j = i \\
\end{array}
\right.
\end{align*}
where the fifth, sixth and ninth equalities follows from equations \eqref{eqn:cri-pi}, \eqref{eqn:cri} and \eqref{eqn:ret}, respectively. 
This entrywise formula is equivalent to our first constraint
$\bbG \L = I - \bone \pi^{\top}.$
\end{proofof}

We close this section by verifying that our formula is consistent with some  facts established in \cite{chung+yau} about Green's function for an undirected graph. First,  the matrix $D^{1/2} \bbG D^{-1/2}$ is symmetric, or equivalently: $\pi_i \bbG(i,j) = \pi_j \bbG(j,i)$ for all $i,j$. Using our formula, indeed we have
\begin{equation}
\label{eqn:green-sym}
\pi_i \bbG(i,j) = \pi_i \pi_j ( H(\pi, j) - H(i,j)) = \pi_i \pi_j ( H(\pi, i) - H(j,i)) = \pi_j \bbG(j,i)
\end{equation}
by equation \eqref{eqn:cri-pi}. Second, we have
$$
H(i,j) = \frac{\vol(G)}{\deg(j)} \bbG(j,j) - \frac{\vol(G)}{\deg(j)}  \bbG(i,j) = \frac{1}{\pi_j} (\bbG(j,j) - \bbG(i,j)).
$$
Note that the equation above corrects a typo in the formulation (Theorem 8) in \cite{chung+yau}, where the  coefficient for $\bbG(i,j)$ is written as $\vol(G)/\deg(i)$ rather than $\vol(G)/\deg(j)$. (This typo originates with a variation of this minor error in their equation (23).) Using our equation for $\bbG(i,j)$, we have
$
\frac{1}{\pi_j} (\bbG(j,j) - \bbG(i,j)) = H(\pi,j) - (H(\pi,j) - H(i,j)) = H(i, j). 
$

\section{Proof for directed graphs via  exit frequencies}
\label{sec:efm-proof}

In this section, we give a second proof of Theorem \ref{thm:green} that reveals the connections between $\bbG$ and the theory of exact stopping rules for random walks on graphs. Our main contribution is in recognizing the deep relationship between these two lines of research. Once we have developed the proper context, our proof of Theorem \ref{thm:green} will be quite short.
We start with an overview of the results in Lov\'asz and Winkler \cite{lovasz+winkler2}. 

Let $G$ be a digraph with transition matrix $P$, so that $\L = I - P$.
Given any starting distribution $\sigma$ and any target distribution $\tau$, there exist one or more stopping rules $\Gamma:\sigma \rightarrow \tau$ that generate a sample from $\tau$ when started from a vertex drawn from $\sigma$. Such a rule $\Gamma$ is \emph{optimal} when it minimizes the expected length of the rule $E(\Gamma)$ among all $(\sigma,\tau)$-rules. The \emph{access time} $H(\sigma, \tau)$ is  the expected length of an optimal $(\sigma,\tau)$-rule,
\[
H(\sigma,\tau) = \min_{\Gamma:\sigma \rightarrow \tau} E(\Gamma).
\] 
For example, when $\sigma=i$ and $\tau=j$ are singleton distributions, the access time equals the hitting time $H(i,j)$. 

The key to understanding an access time is to partition this expected length by the vertices of the graph. Let $\Gamma$ be an optimal $(\sigma,\tau)$-rule. We define the $k$th \emph{exit frequency} $x_k(\sigma,\tau)$ to be the expected number of exits from vertex $k$ when following rule $\Gamma$. The \emph{conservation equation} \cite{pitman} states that for all $j \in V$,
\begin{equation}
\label{eqn:conservation}
\sum_{i \in V} p_{ij} x_i (\sigma, \tau) - x_j (\sigma,\tau) = \tau_j - \sigma_j.
\end{equation}
Intuitively, this says that the expected difference between the number of entrances and exits at $j$ must be $\tau_j - \sigma_j$. All optimal $(\sigma,\tau)$-rules have the same exit frequencies (even though the rules themselves may have very different execution). Moreover, we have the following simple test for optimality:
\[
\Gamma \mbox{ is an optimal stopping rule} \Longleftrightarrow \exists k \in V, x_k(\Gamma)=0.
\]
Such a vertex $k$ with $x_k(\sigma,\tau)=0$ is called a $(\sigma,\tau)$-\emph{halting state}. This simple criterion makes it easy to determine whether a stopping rule is optimal: we must simply check whether there is a vertex that is never exited. For example, when our target is a singleton $\tau=j$, the rule ``walk until you reach $j$'' is an optimal rule (since $j$ is a halting state), so the access time from $\sigma$ to $j$  is $H(\sigma,j) = \sum_{i \in V} \sigma_i H(i, j).$

For our final result from \cite{lovasz+winkler},  we have a formula for optimal exit frequencies in terms of access times:
\begin{equation*}
x_j(\sigma,\tau) = \pi_j (H(\sigma, \tau) + H(\tau, j) - H(\sigma, j)).
\end{equation*}
In other words, the $j$th exit frequency measures the expected penalty for obtaining a sample from $\tau$ along the way during a stopping rule from $\sigma$ to $j$.  Herein, we  focus on stopping rules that  start from a singleton distribution $\sigma=i$, in which case our formula is
\begin{equation}
\label{eqn:exit-freq2}
x_j(i,\tau) = \pi_j (H(i, \tau) + H(\tau, j) - H(i, j)).
\end{equation}

Next, we adopt the matrix viewpoint introduced in \cite{beveridge+lovasz}.
  Fixing the target distribution $\tau$, we consider the family of stopping rules from singletons to $\tau$ as an ensemble. We create an $n \times n$ matrix $X_{\tau}$ whose $i$th row contains the exit frequencies for an optimal $(i, \tau)$-stopping rule. This gives us the \emph{exit frequency matrix} whose $ij$th entry is 
$$
(X_{\tau})_{ij} = x_j (i, \tau).
$$
We can then write the conservation equation \eqref{eqn:conservation} for this ensemble of optimal rules in matrix form:
\begin{equation}
\label{eqn:cons-matrix}
X_{\tau} \L = X_{\tau} (I-P) = I - \bone \tau^{\top}.
\end{equation}
We are now ready for the second proof of our main result.

\begin{proofof}{Theorem \ref{thm:green}}
Consider the exit frequency matrix $X_{\pi}$. By equation \eqref{eqn:cons-matrix}, we have $X_{\pi} \L = I - \bone \pi^{\top}$. This is the first constraint of \eqref{eqn:green-cons} for Green's function. Since $\L$ has rank $n-1$, and $\pi$ is a left eigenvector of $\L$ for eigenvalue 0, we have
\begin{equation}
\label{eqn:green-efm}
\bbG = X_{\pi} - \bh \pi^{\top}
\end{equation}
where $\bh$ is some constant vector.
The second constraint in equation \eqref{eqn:green-cons} requires that the rows of $\bbG$ all sum to zero. Meanwhile, the $i$th row of $X_{\pi}$ sums to $H(i,\pi)$:  adding up all the expected exits gives the expected length of the rule. Therefore
 $h_i = H(i,\pi)$, so that the $ij$th entry of Green's function is
\[
\bbG(i,j) = x_j(i,\pi) - \pi_j H(i,\pi) = \pi_j (H(\pi, j) - H(i,j))
\]
by equation \eqref{eqn:exit-freq2} with target distribution $\tau=\pi$.
\end{proofof}

Equation \eqref{eqn:green-efm} reveals that Green's function is the exit frequency matrix $X_{\pi}$ plus a rank one matrix. Further investigation of this exit frequency matrix can be found in \cite{beveridge+lovasz}.   We conclude this section with some immediate consequences of equation \eqref{eqn:green-efm}. In Section \ref{sec:green-general}, we pursue some duality results for Green's function, analogous to those found in \cite{beveridge+lovasz}.

First, we confirm that typically $\pi_i \bbG(i,j) \neq \pi_j \bbG(j,i)$
 because the cycle reversing identity \eqref{eqn:cri}  does not hold for  digraphs. Next, we 
give an alternative definition for the exit frequency matrix $X_{\pi}$ that parallels the Green's function constraints \eqref{eqn:green-cons}. 
As noted above, every optimal stopping rule contains a halting state. Therefore $X_{\pi}$ is the unique matrix satisfying
\begin{equation}
\label{eqn:efm-cons}
\begin{array}{rcl}
X_{\pi} \L &=& I - \bone \pi^{\top}, \\
\min_{j \in V} (X_{\pi})_{ij}  &=& 0 \quad 1 \leq i \leq n.
\end{array}
\end{equation}
The second constraint for $X_{\pi}$ requires that all entries are nonnegative, and at least one entry in each row is zero. As noted above, this is equivalent to saying that  the entries in the $i$th row are the exit frequencies for an \emph{optimal} stopping rule from $i$ to $\pi$. 
The second constraint \eqref{eqn:green-cons} on $\bbG$ makes sense from a linear algebraic perspective, while the second constraint \eqref{eqn:efm-cons} on $X_{\pi}$  is  fundamental to the stopping rule point of view.

Finally, we make some additional connections between Green's function and the theory of stopping rules. 
First, it is clear that for all $j \in V$, we have
$$H(\pi, j) =  \frac{1}{\pi_j} \bbG(j,j)$$
and that
$\Thit = \mathrm{Tr}(\bbG).$ The latter observation is also a manifestation of the spectral identity $\Thit= \sum_{k=1}^n 1/ \lambda_k$ listed in Theorem \ref{thm:green-spectral}.
Second, 
for every row $i$, we have
$$\displaystyle{\sum_{j \in V} \bbG(i,j) = - H(i, \pi).}$$
Therefore we can obtain the reset time by taking the weighted sum of these row sums:
\[
\Tres =  - \sum_{i \in V}  \sum_{j \in V} \pi_i \bbG(i,j).
\]
Third, we can also recover the mixing time starting from $i$ as follows:
\begin{equation}
\label{eqn:imix-from-green}
H(i, \pi) =  \max_{j} \, - \frac{1}{\pi_j} \bbG(i,j).
\end{equation}
Equivalently, $H(i,\pi)$ is the largest entry in the $i$th row of the matrix product $- \bbG \, \Pi^{-1} $ 
where $\Pi^{-1} = \mbox{diag}(\pi_1^{-1}, \pi_2^{-1}, \ldots, \pi_n^{-1}).$ In other words,
$$
H(i, \pi) = \max_{j \in V} \, - (\bbG \, \Pi^{-1})_{ij}.
$$
This brings us to our final observation: the mixing time $\Tmix$ is 
\begin{equation}
\label{eqn:tmix-from-green}
\Tmix = \max_{i \in V}  \max_{j \in V} \, - (\bbG \, \Pi^{-1} )_{ij}.
\end{equation}
In summary, these quantities are easily accessible, once we have Green's function. This underscores the deep connection between this pseudo-inverse of the Laplace operator and the theory of exact stopping rules.

\section{Examples}
\label{sec:examples}

In this section, we calculate Green's function for some  families of undirected graphs. The formulas for the complete bipartite graph and for general trees are new. The remaining formulas have previously been determined via different methods, as noted below. Our calculations are faster, and the connections to stopping rules provide new insight into many of these values.
Before getting to the examples, we must recount some results from Lov\'asz and Winkler \cite{lovasz+winkler-forget} about optimal mixing rules. That paper considers directed graphs, but we only present the simpler formulations  for the undirected graph case. 

Halting states for an optimal rule $\Gamma(i,\pi)$ enjoy some additional structure. 
 Given a target vertex $i$, another vertex $j$ is called \emph{$i$-pessimal} if it achieves $H(j,i) = \max_{k \in V} H(k,i)$. We use $i'$ to denote an $i$-pessimal vertex. (There may be multiple $i$-pessimal vertices; in this case we take $i'$ to be an arbitrarily chosen one.) For an undirected graph, we have two different formulas for the mixing time:
\begin{equation}
\label{eqn:mixing-time}
H(i, \pi) = H(i, i') - H(\pi, i').
\end{equation}
and
\begin{equation}
\label{eqn:mixing-time2}
H(i, \pi) = H(i', i) - H(\pi, i). 
\end{equation}
Furthermore,  $i'$ is a halting state for an optimal $(i,\pi)$-stopping rule.
A vertex $z$ such that $H(z,\pi) = \Tmix$ is called \emph{mixing pessimal.} 
We have the following useful equivalence for a vertex $z \in V$ on an undirected graph:
\begin{equation}
\label{eqn:tmix-achieved}
 \Tmix = H(z,\pi) \quad \Longleftrightarrow \quad H(z',z) = \max_{i \in V} H(i',i).
\end{equation}
For example, the endpoints of the path $P_n$ are both mixing pessimal. More generally, 
 if $z$ is mixing pessimal, then so is $z'$, meaning that $H(z,\pi) = \Tmix = H(z', \pi)$. Moreover, $z$ is a halting state for an optimal $(z',\pi)$-stopping rule.

{\bf The Complete Graph.} \quad
Green's function for $K_n$ was calculated by Xu and Yau \cite{xu+yau}. We use equation \eqref{eq:green-formula} to find these values.  It is easy to verify that $H(i,j) = n-1$ for all $i \neq j$. Therefore
\begin{align*}
\bbG(i,j) = {\pi_j} (H(\pi,j) - H(i,j)) = \frac{1}{n} \left( 
\frac{n-1}{n} \cdot (n-1) - H(i,j)
 \right)
 =
 \left\{
 \begin{array}{cl}
 -\frac{n-1}{n^2} & i \neq j,\\
 \left(\frac{n-1}{n}\right)^2 & i = j.
 \end{array}
 \right.
\end{align*}

{\bf The Complete Bipartite Graph and the Star.} \quad
Consider the complete bipartite graph $K_{r,s}$ where $|U|=r$ and $|W|=s$ are the partite sets. A simple calculation shows that 
for $u_i, u_j \in U$ and $w_k, w_{\ell} \in W$ where $i \neq j$ and $k \neq \ell$,
\[
\begin{array}{ccc}
H(u_i,w_k) = 2s-1, &\quad&
H(w_k,u_i) = 2r-1, \\
H(u_i,u_j) = 2r,  &&
H(w_k,w_\ell) = 2s. 
\end{array}
\]
As for access times from the stationary distribution, we have
\[
\begin{array}{ccc}
H(\pi,u) = 2r - \frac{3}{2}, &\quad&
H(\pi,w) = 2s- \frac{3}{2}.
\end{array}
\]
By equation \eqref{eq:green-formula}, we have
\[
\begin{array}{ccc}
\begin{array}{rcl}
\bbG(u_i,u_i) &=& 1 - \frac{3}{4r}, \\
\bbG(u_i,u_j) &=&   -\frac{3}{4r}, \\
\bbG(u_i,w_k) &=&  -\frac{1}{4s},  
\end{array}
&\quad&
\begin{array}{rcl}
\bbG(w_k,w_k) &=& 1 - \frac{3}{4s},  \\
\bbG(w_k,w_{\ell}) &=& -\frac{3}{4s}, \\
\bbG(w_k,u_i) &=& -\frac{1}{4r}.  \\
\end{array}
\end{array}
\]
In the special case of the star $K_{1,n-1}$ with center $c$ and leaves $v,w$ are leaves, we have
\[
\begin{array}{ccc}
\begin{array}{ccc}
\bbG(c,c) &=& \frac{1}{4}, \\
\bbG(c,v) &=& - \frac{1}{4(n-1)} ,
\end{array}
&\quad&
\begin{array}{ccc}
\bbG(v,v) &=& 1 - \frac{3}{4 (n-1)}, \\
\bbG(v,c) &=& - \frac{1}{4},  \\
\bbG(v,w) &=& - \frac{3}{4(n-1)}.
\end{array}
\end{array}
\]
The Green's function values for $K_{1,n-1}$ were also calculated in \cite{xu+yau}.

{\bf The Path.} \quad
Green's function for the path $P_n$ was calculated in \cite{bcsl} using a formula for the normalized Green's function for $P_n$ from \cite{chung+yau}. We derive the same formula, and provide additional insight into its component terms. 
For $1 \leq i \leq j \leq n$, we have
\begin{align}
\nonumber
 \bbG(i,j)  &=  \pi_j \big( H(\pi,j) - H(i,j)  \big) \\
 \nonumber
&= \pi_j \big( (H(\pi,j) - H(i,j)) + ( H(1,n) - H(\pi,n)) - (H(1,n) - H(\pi,n)) \big) \\
\nonumber
&= \pi_j \big((H(\pi,j) - H(\pi,n)) + (H(1,i) + H(j,n)) - \Tmix(P_n) \big)\\
\nonumber
&= \pi_j \big((H(n,j) - H(j,n))+ (H(1,i) + H(j,n)) - \Tmix(P_n) \big) \\
\label{eqn:green-path}
&= \pi_j \big(H(1,i) + H(n,j) - \Tmix(P_n) \big),
\end{align}
where the third equality follows from equation \eqref{eqn:mixing-time} and the fourth follows from equation \eqref{eqn:cri-pi}.
The value of $\Tmix(P_n)$ for the path was calculated in \cite{beveridge+wang}, and using this value we obtain
$$ \bbG(i,j)= \pi_j \left( (i-1)^2 + (n-j)^2 - \frac{2n^2-4n+3}{6} \right)
$$
which matches the formula in \cite{bcsl}. Note that equation \eqref{eqn:green-path} tells a compelling story about what the value of $\bbG(i,j)$ captures about the graph. It sums  the hitting times from the ends of the path to $i$ and $j$, respectively, and then subtracts the mixing time.

{\bf Trees.} \quad
The path formula \eqref{eqn:green-path} for Green's function can be generalized to an arbitrary tree $G$. 
Let $z$ and $z'$ be vertices that achieve $H(z',z) = \max_i H(i',i).$  By equation \eqref{eqn:tmix-achieved}, we have  $H(z,\pi) = \Tmix = H(z', \pi)$. Let $W = \{z = w_1, \ldots, w_k = z' \}$ be the unique $(z,z')$-path in $G$. We now calculate $\bbG(i,j)$. Let $i^*=w_r$ (resp.~$j^*=w_s$) be the vertex in $W$ that is closest to $i$ (resp.~$j$). Without loss of generality, assume that $r \leq s$. Figure \ref{fig:green-tree} shows an example. 
We  manipulate $\bbG(i,j)$ as we did in the path example above to obtain
\begin{align*}
\frac{1}{\pi_j} \bbG(i,j) 
&= \pi_j \big( H(\pi,j) - H(i,j) \big)\\
&= \pi_j \big( (H(\pi,j) - H(i,j)) + ( H(z,z') - H(\pi,z')) - (H(z,z') - H(\pi,z')) \big) \\
&= \pi_j \big( (H(\pi,j) - H(\pi,z')) + (H(z,z') - H(i,j)) - \Tmix(G)  \big) \\
&= \pi_j \big( (H(z',j) - H(j,z'))+ (H(z,z') - H(i,j)) - \Tmix(G) \big) \\
&= \pi_j  \bigg( \big( H(z',j^*) - H(j,j^*) \big) +  \big( H(z,i^*) - H(i, i^*) \big) - \Tmix(G) \bigg)
\end{align*}
where the last equality follows from expanding hitting times such as $H(z',j) = H(z', j^*) + H(j^*,j)$. In the case of the path, $i^*=i$ and $j^*=j$ while $z=1$ and $z'=n$, which gives us the path formula above. Further investigation of exact stopping rules on trees can be found in \cite{beveridge, beveridge+wang, beveridge+youngblood}.

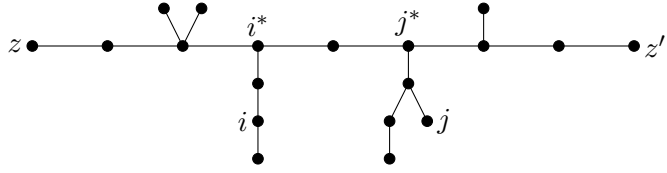
\begin{figure}[t]

\begin{center}
\begin{tikzpicture}

\draw (0,0) -- (8,0);

\foreach \i in {0,1,2,3,4,5,6,7,8}
{
\draw[fill] (\i, 0) circle (2pt);
}

\draw (1.75, .5) -- (2,0) -- (2.25, .5);

\draw[fill] (1.75, .5) circle (2pt);
\draw[fill] (2.25, .5) circle (2pt);

\draw (3,0) -- (3,-1.5);

\draw[fill] (3, -.5) circle (2pt);
\draw[fill] (3, -1) circle (2pt);
\draw[fill] (3, -1.5) circle (2pt);

\node[left] at (0,0) {$z$};
\node[right] at (8,0) {$z'$};

\node[left] at (3,-1) {$i$};
\node[above] at (3,0) {$i^*$};

\node[right] at (5.25,-1) {$j$};
\node[above] at (5,0) {$j^*$};

\draw (4.75, -1.5) -- (4.75, -1) -- (5, -.5) -- (5.25, -1);
\draw (5,0) -- (5, -.5);

\draw[fill] (5, -.5) circle (2pt);
\draw[fill] (5.25, -1) circle (2pt);
\draw[fill] (4.75, -1) circle (2pt);
\draw[fill] (4.75, -1.5) circle (2pt);

\draw (6,0) -- (6,.5);
\draw[fill] (6,.5) circle (2pt);

\end{tikzpicture}
\end{center}

\caption{A tree $G$ with mixing pessimal vertices $z$ and $z'$ where $H(z',z) = \max_{v \in V} H(v,z)$. The vertices $i$ and $j$ are such that the $(i,z)$-path intersects the $(z,z')$ path no further from $z$ than the $(j,z)$-path does.}

\label{fig:green-tree}

\end{figure}

{\bf The Cycle.} \quad Ellis \cite{ellis} contains a formula for Green's function on the cycle. We give a new derivation. 
We take our vertices to be $\{0,1,\ldots, n-1 \}$ and calculate the values for vertex $0$. It is well known (and easy to check) that $H(0,j) = H(j,0) = j(n-j).$
Therefore
\[
H(\pi,0) = \sum_{j} \frac{j(n-j)}{n} =  \frac{(n+1)(n-1)}{6}.
\]
We then find that
\[
\bbG(0,j)= \frac{1}{n} \left(  \frac{(n+1)(n-1)}{6} - j(n-j) \right).
\]

{\bf The Hypercube.} \quad
Let $Q_d$ be the $d$-dimensional hypercube, where the vertices are labeled by binary $d$-tuples $\bv = (v_1, v_2, \ldots , v_d )$. Green's function for the hypercube was calculated in \cite{chung+yau}. We provide a simpler formula, and make some further observations.

We start by finding hitting time formulas to vertex $\bzero$, originally calculated by Pomerance and Winkler \cite{winkler}. 
For $0 \leq \ell \leq d$, let $V_{k}$  denote the $k$th level of $Q_d$, consisting of all vertices labelled with $k$ ones. 
 Let $T_k$ denote the expected time for a random walk started at $v \in V_k$  to reach $V_{k-1}$. Clearly, $T_0=0$ and the remaining level-wise hitting times satisfy the recurrence
$$
T_k = 1 + \frac{d-k}{d} ( T_{k+1} + T_k ), \quad 1 \leq k \leq d.
$$
It is easy to check that for $0 < k \leq d$, we have
$$
T_k = \frac{\sum_{j=k}^n {d \choose j}}{{d-1 \choose k-1}}.
$$
Therefore the hitting time from a vertex $\bv \in V_{\ell}$ to $\bzero$ is 
$$
H(\bv, \bzero) = \sum_{k=1}^{\ell} T_k = \sum_{k=1}^{\ell}  \frac{1}{{d-1 \choose k-1}} \sum_{j=k}^d {d \choose j}
= d \sum_{k=1}^{\ell} \frac{1}{k} \sum_{j=k}^d \frac{{d \choose j}}{{d \choose k}}.
$$

The hypercube has a transitive automorphism group, so $H(\pi, i) = \sum_{j \in V} \pi_j H(j,i) = 
\sum_{j \in V} \pi_j H(i,j) = \Thit$. Equation (5.68) in Aldous and Fill \cite{aldous+fill} states that
\begin{equation}
\label{eqn:cube-hit}
\Thit = \frac{d}{2} \sum_{k=1}^d \frac{1}{k} {d \choose k}.
\end{equation}
Therefore, Green's function for the hypercube is given by
\[
\bbG(\bzero, \bv) = \frac{1}{2^d} \left(
\frac{d}{2} \sum_{k=1}^d \frac{1}{k} {d \choose k}
- d \sum_{k=1}^{\ell} \frac{1}{k} \sum_{j=k}^d \frac{{d \choose j}}{{d \choose k}}
\right)
\]
where $\bv \in V_{\ell}$.

We conclude this section by calculating $\Tmix(Q_d)$ using formula \eqref{eqn:mixing-time2}. 
Equation (5.69) in Aldous and Fill \cite{aldous+fill} gives an alternate formula for the pessimal hitting time:
$$
H(\bone, \bzero) = 2^{d-1} \sum_{k=0}^{d-1} \frac{1}{{d-1 \choose k}}.
$$
We use induction to show that $H(\bone, \bzero) = \frac{d}{2} \sum_{k=1}^d \frac{2^k}{k}$. This holds for $d=1$, and we have
\begin{align*}
2^{d} \sum_{k=0}^{d} \frac{1}{{d \choose k}} 
&=
2^{d} \cdot \frac{1}{2} \left( 2 + \sum_{k=0}^{d-1}  \frac{1}{{d \choose k}} + \frac{1}{{d \choose k+1}}\right) 
\, = \,
2^d + 2^{d-1} \left(
\frac{d+1}{d}\sum_{k=0}^{d-1}  \frac{1}{{d-1 \choose k}}
\right) \\
&=  
2^d + \frac{d+1}{d} \cdot \frac{d}{2} \sum_{k=1}^d \frac{2^k}{k}
\, = \,
\frac{d+1}{2} \sum_{k=1}^{d+1} \frac{2^k}{k}.
\end{align*}
Next, a straight-forward inductive argument  shows that
$$
H(\bone, \bzero) = 
\frac{d}{2} \sum_{k=1}^d \frac{2^k}{k} =
\frac{d}{2} \sum_{k=1}^d \frac{1}{k} \left( 1+ {d \choose k} \right)
$$
Combining this final expression with equation
\eqref{eqn:cube-hit}, we find that
$$
\Tmix(Q_d) = -2^d \, \bbG(\bone, \bzero) = H(\bone, \bzero) - \Thit  
=  \frac{d}{2} \sum_{k=1}^d \frac{1}{k} \approx \frac{1}{2} d \log d.
$$
This result squares with the more traditional definitions of the (approximate) mixing time (cf. \cite{aldous+fill}), which are known to be $O(d \log d)$. It is worth remarking that  both $\Thit$ and $H(\bone, \bzero)$ are $O(2^d + 1/d + O(1/d^2))$, while their difference is essentially $ \frac{1}{2} d \log d$. More generally, it is interesting to keep in mind that when a graph has a transitive automorphism group,  the difference between $H(i',i)$ and $\Thit$ is, in fact, the mixing time $\Tmix$.

\section{Spectral formulas}
\label{sec:spectral}

In this section, we explore Green's function, hitting times, and exact mixing measures from the spectral point of view. We prove Theorem \ref{thm:green-spectral} and apply these results to the family of toric grids $C_{n_1} \times C_{n_2} \times \cdots \times C_{n_d}$. 

 Let $\bbL = D^{-1/2} L D^{1/2} = D^{-1} \cL D$. The matrices $\bbL, L, \cL$ all share the same eigenvalues  $0 = \lambda_0 < \lambda_1 \leq \cdots \leq \lambda_{n-1}$, though their eigenbases are different. 
Let $\phi_0, \phi_1, \ldots, \phi_{n-1}$ be the corresponding orthonormal eigenbasis for the symmetric matrix $\bbL$.
We have $\bbL = I - N$ where $N=D^{1/2}A D^{1/2} = D^{-1/2} P D^{1/2}$, so we can reformulate Theorem 3.1 in \cite{lovasz} as
\begin{equation}
\label{eqn:hit-spectral}
H(i,j) =  {\vol(G)} \sum_{k=1}^{n-1} \frac{1}{\lambda_k} 
\left(
\frac{\phi_{kj}^2}{\deg(j)} - \frac{ \phi_{ki} \phi_{kj}}{\sqrt{\deg(j)\deg(i)}} 
\right).
\end{equation}
We use this spectral hitting time formula to find a corresponding formula for the access time from $\pi$ to a singleton distribution, and then prove Theorem \ref{thm:green-spectral}.

\begin{lemma} For  $j \in V$, we have
\begin{equation}
\label{eqn:from-pi-spectral}
H(\pi, j) =   \frac{\vol(G)}{ \deg(j)} \sum_{k=1}^{n-1} \frac{\phi_{kj}^2 }{ \lambda_k}.
\end{equation}
\end{lemma}

\begin{proof}
Using equation \eqref{eqn:hit-spectral}, we obtain
\begin{eqnarray*}
\sum_{i \in V} \pi_i H(i,j) 
&=&
\sum_{i \in V} \sum_{k=1}^{n-1} \frac{1}{\lambda_k} 
\left(
\frac{\deg(i) }{\deg(j)} \phi_{kj}^2 - \sqrt{\frac{\deg(i)}{ \deg(j)}} \phi_{ki} \phi_{kj}
\right) \\
&=&
  \frac{ \vol(G)}{\deg(t)}  \sum_{k=1}^{n-1} \frac{\phi_{kj}^2 }{\lambda_k} 
-  \sum_{k=1}^{n-1} \frac{1}{\lambda_k}
\sum_{i \in V}
\left(
\sqrt{\frac{\deg(i)}{ \deg(j)}} \phi_{ki} \right)  \phi_{kj} \\
&=&
  \frac{\vol(G)}{ \deg(j)} \sum_{k=1}^{n-1} \frac{\phi_{kj}^2 }{ \lambda_k}.
\end{eqnarray*}
The final equality holds because the vector $\bv$ with components $v_{i} = \sqrt{\deg(i)/ \deg(j)}$ is an eigenvector of $\bbL$ for eigenvalue $\lambda_0=0$, and therefore $\bv$ is orthogonal to $\phi_k$ for $1 \leq k \leq n-1$. 
\end{proof}

\begin{proofof}{Theorem \ref{thm:green-spectral}}
First we consider Green's function. Equation \eqref{eqn:green-spectral}
arises by substituting equations \eqref{eqn:hit-spectral} and \eqref{eqn:from-pi-spectral} into formula \eqref{eq:green-formula}.
Next, we give a spectral formula for $H(i,\pi)$.
By equation \eqref{eqn:imix-from-green}, we have 
$H(i, \pi) = \max_{j \in V}  \, - \frac{1}{\pi_j} \bbG(i,j) = \max_{j \in V} (H(i, j) -  H(\pi,j)).$
By equation \eqref{eqn:mixing-time2}, this maximum is achieved precisely when $j$ is an $i$-pessimal vertex. In other words,
\[
H(i,\pi)  = - \frac{\vol(G)}{ \sqrt{\deg(i)\deg(i')}} \,
\sum_{k=1}^{n-1} \frac{1}{\lambda_k} 
 \phi_{ki} \phi_{k i'}
\]
 The formulas \eqref{eqn:mix-spectral} and \eqref{eqn:reset-spectral} for $\Tmix$ and $\Tres$  follow directly.
\end{proofof}

As an application, we evaluate these spectral formulas for products of cycles. Ellis \cite{ellis} gives an alternate, recursive formula for Green's function on this graph family that uses Chebyshev polynomials. As a warm-up, we start with the cycle $C_n$, whose calculation also appears in \cite{ellis}.
 Let $\epsilon = e^{2 \pi i /n}$. The eigenvalues of  $\bbL$  are $\lambda_0, \lambda_1, \ldots, \lambda_{n-1}$ where
$\lambda_k =  1 - \frac{1}{2} (\epsilon^k + \epsilon^{-k})= 1 - \cos (2 \pi k /n)$
with corresponding eigenvector $\phi_k$ whose $i$th component is
$
\phi_k (i) =  \epsilon^{ki} / {\sqrt{n}}.
$
Therefore, Green's function for $C_n$ is given by
\[
\bbG(0,j) = \frac{1}{n} \sum_{k=1}^{n-1} \frac{\epsilon^{kj}}{1 - \cos(2 \pi k / n)} =  
\frac{1}{n} \sum_{k=1}^{n-1} \frac{\cos(2 \pi k j / n)}{1 - \cos(2 \pi k / n)}.
\]
Of course, we also have $\bbG(0,j) = \frac{1}{n} \left( (n+1)(n-1)/6 - j(n-j) \right)$, as discussed in Section \ref{sec:examples}.

We now consider the product of cycles $C_{n_1} \times C_{n_2} \times \ldots \times C_{n_d}$, using $\bx = (x_1, x_2, \ldots, x_d)$ to denote a vertex. We will label the eigensystem using an analogous toric notation $\lambda_{\br} = \lambda_{r_1 r_2 \cdots r_d}.$
The eigenvalues for this Cartesian product can be calculated using Exercise 11.7 in \cite{lovasz-problems}.  Let
$\epsilon_j = e^{2 \pi i / n_j}$ for $1 \leq j \leq d$, and let 
$\br = (r_1, r_2, \ldots, r_d)$ where $0 \leq r_i < n_i$ and $1 \leq i \leq d$. The eigenvalues are
\[
\lambda_{\br} = 
1 - \frac{1}{2d} \sum_{t=1}^{d} (\epsilon_t^{r_t} + \epsilon_t^{-r_t} ) = 1- \frac{1}{d} \sum_{t=1}^{d} \cos \bigg( \frac{2 \pi r_t}{n_t} \bigg).
\]
A corresponding orthonormal eigenbasis for $\bbL$ consists of the vectors $\phi_{r_1 r_2 \cdots r_d}$ given by
\[
\phi_{r_1 r_2 \cdots r_d} (j_1, j_2, \ldots, j_d) = \frac{1}{\sqrt{n}} \epsilon_1^{r_1 j_1} \epsilon_2^{r_2 j_2} \cdots \epsilon_d^{r_d j_d}
\]
where $n = n_1 n_2 \cdots n_d.$
By the symmetry of this toric grid, it is sufficient to calculate Green's function when the first vertex is $\bzero = (0, 0, \ldots, 0)$. We have
$$
\bbG(\bzero, \bj) 
=
\frac{1}{n}
\sum_{\br \neq \bzero} \frac{\epsilon_1^{r_1 j_1} \epsilon_2^{r_2 j_2}  \cdots \epsilon_k^{r_k j_d} }
{1-  d^{-1} \, \sum_{t=1}^{d} \cos ( 2 \pi r_t /n_t )} 
=
\frac{1}{n}
\sum_{\br \neq \bzero} \frac{\cos \left(2 \pi  \sum_{t=1}^d r_t j_t / n_t \right) }
{1- d^{-1} \sum_{t=1}^{d} \cos ( 2 \pi r_t /n_t )}.
$$
Symmetry also tells us that  $H(\pi, 0) = \sum_{j \in V} \pi_j H(j,0) = \sum_{j \in V} \pi_j H(0,j) = \Thit$. So we have
\[
\Thit = n \cdot \frac{1}{n} \sum_{\br \neq \bzero} \frac{1}{\lambda_{\br}} =  \sum_{\br \neq \bzero} \frac{1}{1-  d^{-1} \, \sum_{t=1}^{d} \cos ( 2 \pi r_t /n_t )} 
\]
Since this toric grid has a transitive automorphism group, we have
$\Tmix = \Tres$. The vertex $(\ceil{n_1}, \ceil{n_2}, \ldots, \ceil{n_k} )$ is a $\bzero$-pessimal vertex, so that
\[
\Tres = \Tmix = -
\sum_{\br \neq \bzero} \frac{\cos \left(  \pi  \sum_{t=1}^d r_t + 2 \pi \sum_{t=1}^d \frac{r_t (n_t \, \mathrm{mod} \, 2)}{n_t} \right) }
{1- d^{-1} \sum_{t=1}^{d} \cos ( 2 \pi r_t /n_t )}.
\]
In particular, this formula holds for the hypercube $Q_d$ (where each component is actually the path $P_2$). Simplifying yields the  hypercube formulas as found in Section \ref{sec:examples}.
We  leave these calculations to the interested reader.

\section{A Generalized Green's Function}
\label{sec:green-general}

Exit frequency matrices $X_{\tau}$ (for general target distributions $\tau$) provide great insight into the theory of exact stopping rules. For example, exit frequency matrices lead to  natural proofs of two  non-trivial time reversal identities for exact mixing measures $\Tmix = \rTmix$ and $\Tres = \rTfor$, where the \emph{reverse forget time} $\rTfor$ is defined below \cite{beveridge+lovasz}. In this final section, we define a generalized Green's function $\bbG_{\tau}$ for a given target distribution $\tau$, and we explore some duality results corresponding to time reversal of random walks. 

In Section \ref{sec:efm-proof}, we showed that Green's function is the exit frequency matrix $X_{\pi}$ altered so that the row sums are zero. In other words, we can think of $\bbG$ as a signed exit frequency matrix (where ``negative exit frequencies'' are allowed). From this vantage point, we  define the generalized Green's function  $\bbG_{\tau}$   be the $n \times n$ matrix with entries
\[
\bbG_{\tau}(i,j) = x_j (i, \tau) - \pi_j H(i, \tau) = \pi_j ( H(\tau, j) - H(i,j)) = \pi_j \bigg( \sum_{k} \tau_k H(k,j) - H(i,j) \bigg). 
\]
More compactly, $\bbG_{\tau} = X_\tau - \mathbf{h} \pi^{\top}$ where $h_i = H(i,\tau)$.
Generalizing the discussion in Section \ref{sec:efm-proof}, this matrix satisfies 
\begin{align*}
\bbG_{\tau} \L &= I - \bone^{\top} \tau, \\
\bbG_{\tau} \bone &= \bzero.
\end{align*}
For example, when the target is the singleton distribution $\tau=k$, we have
$$\bbG_k(i,j) = \pi_j (H(k,j) - H(i,j)).$$
The $k$th row of $\bbG_k$ is all-zero, while the $k$th column satisfies $G_k(i,k) = - \pi_i H(i,k)$.

Naturally, we are interested in finding Green's function other useful target distributions.
Here, we discuss a pair of target distributions that are important for stopping rules: the \emph{reverse forget distribution} and the \emph{$\pi$-core distribution}. These distributions are related to the stationary distribution, and they enjoy a duality relationship, as described below.

For full generality, we consider the weighted digraph case, and we will talk about ``the Markov chain $P$'' rather than  ``the random walk on weighted digraph $G$ with transition matrix $P$.''    The Markov chain $P$ has a corresponding dual chain 
$$\rP = \Pi^{-1} P^{\top} \Pi.$$  
Time reversal converts random walks governed by $P$ into random walks governed by $\rP$. We continue to employ hatted notation to indicate quantities associated with the reverse chain $\rP$. For example, $\rH(i,j)$ is the expected length of a random walk from $i$ to $j$ on the reverse chain, and $\rX_{\pi}$ is the exit frequency matrix for optimal mixing rules on the reverse chain.

Next, we introduce another  mixing measure.  The \emph{forget time} is defined as 
$$\Tfor = \min_{\tau} \max_{i \in V} H(i, \tau).$$ 
Lov\'asz and Winkler \cite{lovasz+winkler-forget} proved the remarkable equality $\Tres = \rTfor$ and that the reverse forget time is achieved uniquely by the distribution $\rmu$ where
\[
\rmu_i = \pi_i \bigg( 1 + \sum_{j \in V} p_{ij} H(j, \pi) - H(i, \pi)
\bigg).
\]
The duality between $\pi$ (on the forward chain) and $\rmu$ (on the reverse chain) was placed in a broader framework in \cite{beveridge+lovasz} via exit frequency matrices. The connection  is summarized by the matrix equation
\[
\rX_{\rmu} = \Pi^{-1} \big(X_{\pi}- \bone \mathbf{b}^{\top} \big)^{\top} \Pi \quad \mbox{where} \quad
b_i = \min_{j \in V} x_j(i,\pi). 
\]
The reverse forget distribution $\rmu$ is also called the \emph{$\pi$-contrast distribution} $\pi^*$. The $\pi$-core $\pi^{**}$ is the distribution whose (forward) exit frequency matrix
is 
$$X_{\pi^{**}} = X_{\pi}- \bone \mathbf{b}^{\top}.$$ 
The $\pi$-core $\pi^{**}$ is fully dual to $\rmu=\pi^*$: their exit frequency matrices satisfy
\begin{equation}
\label{eqn:efm-duality}
\rX_{\mu} =  \Pi^{-1} X_{\pi^{**}}^{\top}  \Pi. 
\end{equation}
Furthermore, the formula for $\pi^{**}$ mirrors that of the forget distribution $\rmu$:
\[
\pi^{**}_i = \pi_i \bigg( 1 + \sum_{j \in V} \rp_{ij} \rH(j, \rmu) - \rH(i, \rmu)
\bigg).
\]

We now calculate Green's function for  $\rmu= \pi^*$ and  $\pi^{**}$ for the reverse chain and forward chain, respectively.
The Green's functions for these distributions exhibit some nice duality properties. First, we consider $\rmu$ on the reverse chain:
\[
\rbbG_{\rmu}(i,j) = \rx_j(i,\rmu) - \pi_j \rH(i,\rmu) = \pi_j ( \rH(\rmu, j) - \rH(i,j)).
\]
We can use equation \eqref{eqn:efm-duality} to get an alternative formula:
\begin{align}
\nonumber
\rbbG_{\rmu}(i,j)   
& = \rx_j(i,\rmu) - \pi_j  \sum_{k \in V} \rx_k (i, \rmu) 
\, =  \, \frac{\pi_j}{\pi_i} x_i(j,\pi^{**}) - \pi_j \sum_{k \in V} \frac{\pi_k}{\pi_i} x_i (k, \pi^{**}) \\
\nonumber
&=
\pi_j \bigg(
H(j, \pi^{**}) + H(\pi^{**}, i) - H(j,i) - \sum_{k \in V} \pi_k ( H(k, \pi^{**}) + H(\pi^{**}, i) - H(k,i))
\bigg) \\
\nonumber
&=
\pi_j \bigg(
H(\pi, i) - H(j,i)  +  \bigg(  H(j, \pi^{**})  -\sum_{k \in V} \pi_k  H(k, \pi^{**}) \bigg)
\bigg). \\
\label{eqn:green-forget}
&=
\frac{\pi_j}{\pi_i} \bbG(j,i)   + \pi_j  \bigg( H(j, \pi^{**}) - \sum_{k \in V} \pi_k  H(k, \pi^{**}) \bigg). 
\end{align}
In fact, the $\pi$-core has the nice property that $H(i,\pi) = H(i, \pi^{**}) + H(\pi^{**},\pi)$ for all $i \in V$ (see \cite{beveridge+lovasz}), so we also have
\begin{align*}
\rbbG_{\rmu}(i,j) 
= \frac{\pi_j}{\pi_i} \bbG(j,i) + \pi_j \bigg(  H(j, \pi)  - \Tres \bigg). 
\end{align*}
This equality is reminiscent of the symmetry of equation \eqref{eqn:green-sym} in the undirected case.

A similar argument for $\pi^{**}$ on the forward chain $P$ 
gives
\begin{align}
\nonumber
\bbG_{\pi^{**}}(i,j) 
&= \pi_j ( H(\pi^{**}, j) - H(i,j)) \\
&=
\nonumber
\pi_j \bigg(
\rH(\pi, i) - \rH(j,i)  +  \bigg(  \rH(j, \rmu)  -\sum_{k \in V} \pi_k  \rH(k, \rmu) \bigg)
\bigg)
\label{eqn:green-core} \\
&=
\frac{\pi_j}{\pi_i} \rbbG (j,i) +
\pi_j  \bigg(  \rH(j, \rmu)  -\sum_{k \in V} \pi_k  \rH(k, \rmu) \bigg). 
\end{align}
Equations \eqref{eqn:green-forget} and \eqref{eqn:green-core} relate these matrices to the standard Green's functions on the forward and reverse chains. We make two observations. First,  the quantities on the right hand side concern the dual chain, and the roles of $i$ and $j$ reversed. Second, the right hand side contains a correction term that compares the length of rules to the dual distribution. These types of duality relationships are characteristic for stopping rules on the forward and reverse chains. We believe that this point of view will be useful in exploring further properties of Green's function.

 \section{Acknowledgments}
 
This work was completed while the author was a long-term visitor at the Institute for Mathematics and its Applications, during its thematic year on Discrete Structures and Applications. He is grateful to the IMA, and the organizers and participants of this program. Thanks  to Jeannette Jansson for many helpful conversations, and to Peter Winkler for comments on  this manuscript. 
 

\bibliography{green-efm1}

\end{document}